\documentclass[12pt,a4paper]{amsart}
\usepackage{amssymb}
\usepackage{multirow}
\usepackage{graphicx}
\usepackage{amsmath,amsbsy,amsfonts,amsthm,latexsym,amsopn,amstext,amsxtra,euscript,amscd,mathrsfs}
\usepackage{cite}

\theoremstyle{plain}
\newtheorem{theorem}{Theorem}[section]
\newtheorem{corollary}[theorem]{Corollary}

\theoremstyle{definition}

\theoremstyle{remark}
\newtheorem{remark}{Remark}[section]

\numberwithin{equation}{section}

\numberwithin{table}{section}

\numberwithin{figure}{section}

\begin{document}

\title[Basic trigonometric power sums with applications]{Basic trigonometric power sums with applications}


\author{Carlos M. da Fonseca}
\address{Department of Mathematics, Kuwait University, Safat 13060, Kuwait}
\email{carlos@sci.kuniv.edu.kw}

\author{M. Lawrence Glasser}
\address{Department of Physics, Clarkson University, Potsdam, NY 13699-5820, USA}
\email{laryg@clarkson.edu}

\author{Victor Kowalenko}
\address{ARC Centre of Excellence for Mathematics and Statistics of Complex Systems, Department of Mathematics and Statistics,
The University of Melbourne, Victoria 3010, Australia}
\email{vkowa@unimelb.edu.au}




\begin{abstract}
We present the transformation of several sums of positive integer
powers of the sine and cosine into non-trigonometric combinatorial
forms. The results are applied to the derivation of generating
functions and to the number of the closed walks on a path and in a
cycle.
\end{abstract}

\maketitle

\section{Introduction} \label{sec1}

Over the last half-century there has been widespread interest in
finite sums involving powers of trigonometric functions. In $1966$
Quoniam posed an open problem in which the following result was
conjectured
\begin{equation}\label{eq1}
    \sum_{k=1}^{\lfloor n/2\rfloor} 2^{2 m} \cos^{2m} \left( \frac{k\pi}{n+1}\right)=
    (n+1)\binom{2m-1}{m-1}-2^{2m-1}\, ,
\end{equation}
where $m$ is a positive integer and subject to $m<n+1$ \cite{Q1966}.
In this equation $\lfloor n/2 \rfloor$ denotes the floor function of
$n/2$ or the greatest integer less than or equal to $n/2$. A
solution to the above problem was presented shortly after by
Greening et al. in \cite{G1968}. Soon afterwards, there appeared a
problem involving powers of the secant proposed by Gardner
\cite{Ga1969}, which was solved partially by Fisher
\cite{Fi1971,Kl1990} and completely, only recently, in
\cite{dF2015,K2011}. The activity subsided somewhat until such
series arose in string theory in the early 1990's with the work of
Dowker \cite{Do1992}. Subsequently, a surge occurred with  studies
of related sums and identities as evidenced by the work of: (1)
Berndt and Yeap \cite{BY2002} on reciprocity theorems and (2)
Cvijovi\'c and Srivastava \cite{Cv2012,Cv2007} on Dowker and related
sums, while \cite{C2003,CM1999,WZ2007} were motivated by the
intrinsic fascination these sums possess and derived formulas where
the summand was a power of the secant, e.g.,
$$
\sum_{k=0}^{n-1} \sec^{2p} \left( \frac{k\pi}{n}\right)=
n\sum_{k=1}^{2p-1}(-1)^{p+k}\binom{p-1+kn}{2p-1} \sum_{j=k}^{2p-1}\binom{2p}{j+1}\, .
$$

By far, the most interesting sums have been  those trigonometric power sums with
inverse powers of the sine or cosine, since their evaluation invariably involves the zeta function
directly or through related numbers such as the Bernoulli and Euler numbers. A typical example
is the finite sum of powers of the contangent studied by Berndt and Yeap \cite{BY2002},

\begin{equation} \label{eq1a}
\frac{1}{k} \, \sum_{r=1}^{k-1} \cot^{2n} \left( \frac{r\pi}{k}\right)= (-1)^n-(-1)^n 2^{2n}\sum_{\substack{
j_0,j_1,\ldots,j_{2n} > 0 \\j_0+j_1+\cdots+j_{2n}=n}}
k^{2j_0-1}\prod_{r=0}^{2n}\frac{B_{2j_r}}{(2j_r)!}\, .
\end{equation}
Here, $B_j$ denotes the Bernoulli number with index $j\geqslant 0$.
As described in Appendix A, Berndt and Yeap use contour integration
to derive this result, although more recently it has been studied
with the aid of sampling theorems \cite{AA2011}. Unfortunately, \eqref{eq1a}
is, if not incorrect, confusing or misleading because it states that
the $j_i$ cannot equal zero. Yet, some of them are required to be zero
in order to evaluate the polynomials on the rhs. In addition, it does
not matter if any of the $j_i$ are zero since $B_0$ is equal to
unity anyway. A more precise statement of \eqref{eq1a} is
\begin{equation} \label{eq1b}
\frac{1}{k} \, \sum_{r=1}^{k-1} \cot^{2n} \left(
\frac{r\pi}{k}\right)= (-1)^n -(-1)^n 2^{2n}
\sum_{j_1,j_2,\ldots,j_{2n} =0}^{n,n-j_1,\ldots, n-j_s+j_{2n}}k^{2j_{2n}-1}
\prod_{r=1}^{2n} \frac{B_{2j_{r}}}{(2j_r)!}\,
\frac{B_{2(n-j_s)}}{(2(n-j_s))!} \;\;,
\end{equation}
where $j_s = \sum_{i=1}^{2n}j_i$. The main difference between the
two equations is that there is now an upper limit for each sum over
the $j_i$, which is dependent on the number of sums preceding it. As
a consequence, there is no requirement for the $j_0$ index to appear
as a sum as in \eqref{eq1a}. Instead, it has been replaced by
$n-j_s$ in \eqref{eq1b}. Although the above material is incidental
to the material presented in this paper, because of its importance,
the derivation of \eqref{eq1b} is presented in Appendix A together
with a description of how it is to be implemented when determining
specific powers of the sum, which is another issue overlooked in
\cite{BY2002}. In doing so, we give the $n=3$ and $n=4$ forms for
the sum, thereby demonstrating to the reader just how intricate the
evaluation of trigonometric power sums can be.

It should also be mentioned that although the arguments inside the trigonometric power of the finite sums discussed above are
composed of rational numbers multiplied by $\pi$, the actual sequence of numbers has a profound effect on the final value for the
trigonometric power sum. For example, by using a recursive approach, Byrne and Smith \cite{BS1997} have derived the following result
for the same finite sum over powers of cotangent:
$$
\sum_{r=1}^k \cot^{2n}\left( \frac{(r-1/2) \pi}{2 k}\right)= (-1)^k
k+ \sum_{j=1}^n b_{n,j}\, k^{2j}\, ,
$$
where the coefficients are given by
$$
b_{n,j}=\frac{1}{2^{2(n-j)-1}}\sum_{\ell=1}^{n-j} (-1)^\ell
\binom{2n}{\ell} b_{n-\ell,j} \quad \mbox{and}\quad \sum_{j=1}^n
b_{n,j}=1+(-1)^{n-1}\, ,\quad \mbox {for $j<n$.}
$$
Moreover, Byrne and Smith express the $b_{n,j}$ in terms of the
odd-indexed Euler numbers as opposed to the Bernoulli numbers
appearing in \eqref{eq1a} and \eqref{eq1b}. Hence we see that the
sum yields completely different results when the argument inside the
cotangent is altered to $(r - 1/2)\pi/2k$ as opposed to $r \pi/k$ in
\eqref{eq1a} and \eqref{eq1b}.

Although there has been a greater interest in sums with inverse powers of sine or
cosine (including powers of contangent and tangent), there are still many
basic trigonometric power sums that have not been solved. By a basic
trigonometric power sum, we mean a finite sum involving positive powers
of a cosine or sine whose arguments are rational multiples of $\pi$. Such series can
be as difficult to evaluate as their ``inverse power" counterparts even though they
tend to yield combinatorial solutions directly rather than involve the Riemann zeta function
or related quantities (Bernoulli numbers), as in \eqref{eq1a} or \eqref{eq1b},
before ultimately reducing to the simple polynomial solutions  in Appendix A.
As we shall see, just like their inverse power counterparts, the closed form expressions
for basic trigonometric power sums depend greatly, not only on the power of the
trigonometric function, but also on their limits and the values of the rationals
multiplying $\pi$ in the argument.

In response to the situation concerning basic trigonometric power sums, Merca \cite{M2012}
recently derived formulas for various basic cosine power sums including
\begin{equation} \label{merca1}
\sum_{k=1}^{\lfloor (n-1)/2\rfloor} \cos^{2p} \left( \frac{k\pi}{n}\right)=
-\frac 12 +\frac{n}{2^{2p+1}}\sum_{k=-\lfloor p/n\rfloor}^{\lfloor p/n\rfloor}\binom{2p}{p+kn}
\end{equation}
and
\begin{equation} \label{merca2}
\sum_{k=1}^{\lfloor n/2\rfloor} \cos^{2p} \left(\frac{(k-1/2) \pi}{n}\right)=\frac{n}{2^{2p+1}}
\sum_{k=-\lfloor p/n\rfloor}^{\lfloor p/n\rfloor}(-1)^k\binom{2p}{p+kn}
\;,
\end{equation}
where $n$ and $p$ represent positive integers. By using these results he was able to derive
in a series of corollaries several new combinatorial identities involving finite sums over
$k$ of the binomial coefficient $\binom{2m}{n-rk}$, where $r$ is an integer. Following this
work, two of us \cite{FK2013} derived a formula for the basic trigonometric power sum with
the alternating phase factor $(-1)^k$ in the summand and an extra factor of 2 in the denominator
of the argument. The formula was also found to be combinatorial in nature, but its actual form
was different when both the power of the sine or cosine and its upper limit were varied. Consequently,
a computer program was required to evaluate the series in rational form for each set of these values.

In a more recent work using Dickson and Chebyshev polynomials Barbero \cite{B2014} initially obtained
the following result:
\begin{equation} \label{barb1}
R_{m,n} =2^{2m} \sum_{k=1}^{n+1} \cos^{2m} \Bigl( \frac{k \pi}{2n+3}
\Bigr)= \Bigl( n+ \frac{3}{2} \Bigr) \binom{2m}{m}- 2^{2m-1}\;,
\quad m \geqslant 1\,,
\end{equation}
with $R_{n,0}=n+1$ and $R_{0,m}=1$. Though elegant, this result was found not to be entirely correct.
E.g., for $m=12$ and $n=3$, the representation in terms of the trigonometric power sum for $R_{m,n}$
gives a value of $3\,798\,310$, while its combinatorial form on the rhs gives a value of $3\,780\,094$. Yet,
if one replaces $n$ and $p$, respectively, by $2n+3$ and $m$ in \eqref{merca1}, then after multiplying
by $2^{2m}$ one finds that the combinatorial form on the rhs yields the correct value of $3\,798\,310$.
Apparently, the discrepancy between the lhs and rhs of \eqref{barb1} has been brought to Barbero's attention,
since he has amended the original result to
\begin{equation} \label{barb2}
2^{2m} \sum_{k=1}^{n+1} \cos^{2m}\Bigl( \frac{k \pi}{2n+3}\Bigr)
=  \begin{cases} \displaystyle \Bigl( n+ \frac{3}{2} \Bigr) \binom{2m}{m}- 2^{2m-1}\;, \quad 1 \leqslant m < (2n+3)\,, \cr
\displaystyle \Bigl(( n+ \frac{3}{2} \Bigr) \binom{2m}{m}- 2^{2m-1} \cr
\displaystyle + (2n+3) \sum_{i=1}^{\lfloor m/(2n+3) \rfloor} \binom{2m}{m-(2n+3)i}, \quad m \geq 2n+3\;\;, \end{cases}
\end{equation}
with $R_{n,0}$ and $R_{0,m}$ as given above. Consequently, we find
that the extra sum on the rhs of the second result in \eqref{barb2}
yields the discrepancy of 18216 on the rhs of \eqref{barb1} when
$m=12$ and $n=3$. This highlights the necessity for conducting
numerical checks on results rather than solely relying on proofs,
where small terms such as the extra sum in the second result of
\eqref{barb2} can often be neglected.

In this paper we aim to continue with the derivation of combinatorial forms for basic trigonometric
power sums possessing different arguments and limits than those calculated previously. Typically,
the basic trigonometric power sums studied here will be of the form:
$$
S= \sum_{k=0}^{g(n)}  (\pm 1)^k \, f(k)  \left\{ \begin{matrix} \cos^{2 m} \\ \sin^{2m}
\end{matrix} \right\} \left( \frac{qk \pi}{n} \right)  \;\;,
$$
where $m$, $q$ and $n$ are positive integers, $g(n)$ depends upon
$n$, e.g., $n - 1$ or $\lfloor m/n\rfloor$, and $f(k)$ is a simple
function of $k$, e.g.,  unity or $\cos(k \pi/p)$ with $p$, an
integer. Surprisingly, the results presented here will be required
when we study more complicated sums with inverse powers of
trigonometric functions, such as the general or twisted Dowker
\cite{Do1992} and related sums \cite{Cv2012}, in the future.
Furthermore, we shall apply the results of Section \ref{sec2} in the
derivation of generating functions and finally consider an
application to spectral graph theory by determining the number of
closed walks of a specific length on a path and in a cycle.

\section{Main result} \label{sec2}
The main result in this paper is presented in the following theorem:
\begin{theorem} \label{main}
Let $m$ and $n$ be positive integers in the basic trigonometric
power sums
$$
C(m,n):= \sum_{k=0}^{n-1} \cos^{2m}  \left( \frac{k\pi}{n}\right)
\quad \mbox{and} \quad S(m,n):= \sum_{k=0}^{n-1} \sin^{2m}  \left(
\frac{k\pi}{n}\right)\, .
$$
Then it can be shown that
\begin{equation}\label{eq2a}
C(m,n) =
\begin{cases} \displaystyle
2^{1-2m} \, n \Bigl(\binom{2m-1}{m-1} +\sum_{p=1}^{\lfloor
m/n\rfloor}\binom{2m}{m-pn} \Bigr) \, , & m\geqslant n\,, \cr
\displaystyle 2^{1-2m} \, n \binom{2m-1}{m-1}\;, & m<n\,,
\end{cases}
\end{equation}
and
\begin{equation}\label{eq2b}
S(m,n)=
\begin{cases} \displaystyle
 2^{1-2m} \,n
\Bigl(\binom{2m-1}{m-1}+\sum_{p=1}^{\lfloor
m/n\rfloor}(-1)^{pn}\binom{2m}{m-pn} \Bigr)\,, &  m \geqslant n \,,
\cr \displaystyle 2^{1-2m} \,n \binom{2m-1}{m-1} \,, & m<n.
\end{cases}
\end{equation}
\end{theorem}

\begin{remark}
After preparing the manuscript, it came to our attention that
\eqref{eq2a} appears in a more unwieldy form as (18.1.5) in
\cite{Ha1975}, while the second result of \eqref{eq2a} appears as
No. 4.4.2.11 in \cite{PBM2003}. This suggests our proof is entirely
different from these references. Moreover, we shall adapt the proof
to determine other results not given in these references.
\end{remark}

\begin{proof}
We begin by stating well-known trigonometric power sums, which
appear as No. 4.4.2.1 in \cite{PBM2003}. These are
\begin{equation} \label{eq3a}
\sum_{k=1}^{n} \cos^{2m} (kx) = 2^{1-2m}
\sum_{k=1}^{m}\binom{2m}{m-k}\frac{\sin (nkx)}{\sin (kx)} \;
\cos\bigl( (n+1)kx \bigr) +  2^{-2m} \,n \binom{2m}{m}
\end{equation}
and
\begin{equation}\label{eq3b}
\sum_{k=1}^{n} \sin^{2m} (kx) = 2^{1-2m}\sum_{k=1}^{m}(-1)^k
\binom{2m}{m-k}\frac{\sin (nkx)}{\sin (kx)}\; \cos \bigl( (n+1) kx
\bigr) +  2^{-2m}\, n \binom{2m}{m}\, .
\end{equation}

For $x=\pi /n$, \eqref{eq3a} becomes
\begin{equation}\label{eq4}
C(m,n)=2^{-2m}\, n \binom{2m}{m}+ 2^{1-2m} \sum_{k=1}^{m} (-1)^k
\,\binom{2m}{m-k}\, R(k) \cos(k \pi/n) \, ,
\end{equation}
where
\begin{equation*}\label{eq4a}
R(k) =\lim_{y \to \pi} \left\{ \frac{\sin(ky)}{\sin(ky/n)} \right\} = \lim_{\epsilon \to 0}
\left\{ \frac{\sin(k(\pi +\epsilon))}{\sin(k(\pi +\epsilon)/n)} \right\} \;.
\end{equation*}

The quotient of sines given above by $R(k)$ vanishes for all values
of $k$ except when $k$ is a multiple of $n$, i.e., when $k =  pn$,
where $p=1,2,\ldots,\lfloor m/n \rfloor$. For these values of $k$,
we find that $R(k) = (-1)^{(n-1)p} n$. The phase factor in $R(k)$
cancels $(-1)^k \cos(k\pi/n)$ in the summand of $C(m,n)$. Moreover,
since $ \binom{2m}{m}=2\binom{2m-1}{m}$, \eqref{eq4} becomes
\begin{equation*}\label{eq4b}
C(m,n)= 2^{1-2m}\,n \binom{2m-1}{m}+ 2^{1-2m}\,n \sum_{p=1}^{\lfloor
m/n \rfloor} \binom{2m}{m-pn} \;.
\end{equation*}
Re-arranging the terms on the rhs of the above result then yields
the first result in \eqref{eq2a}. For $m<n$, the sum over $p$
vanishes and we are left with the second result in \eqref{eq2a}. In
addition, for $m=0$, the second term not only vanishes, but also the
first term yields $n$. Finally, adopting the same approach to
\eqref{eq3b} yields both results in \eqref{eq2b}. This completes the
proof.
\end{proof}

From Theorem \ref{main}, we can obtain further results beginning with
the following corollary:
\begin{corollary}
For $q\equiv 0 \, ({\rm mod}\, n)$, where $n$ is a positive integer,
the following generalizations of the above basic trigonometric power
sums are given by
\begin{equation}\label{eq4c}
\sum_{k=0}^{q-1} \cos^{2m} \Bigl( \frac{k \pi}{n} \Bigr) =
\frac{q}{n}\,  \,C(m,n)\,,
\end{equation}
and
\begin{equation}\label{eq4d}
\sum_{k=0}^{q-1} \sin^{2m} \Bigl( \frac{k \pi}{n} \Bigr) =
\frac{q}{n} \, S(m,n)  \,,
\end{equation}
where $C(m,n)$ and $S(m,n)$ are obtained from \eqref{eq2a} and
\eqref{eq2b} respectively.
\end{corollary}

\begin{proof}
From the condition on $q$, we let $q=sn$, where $s$ is a positive
integer. Then we note that the basic trigonometric power sum in
\eqref{eq4c} can be subdivided according to
\begin{equation}\label{eq4e}
\sum_{k=1}^{q-1} \cos^{2m} \Bigl( \frac{k \pi}{n} \Bigr) =
\sum_{j=0}^{s-1} \sum_{k=jn +1}^{(j+1)n-1} \cos^{2m} \Bigl(\frac{k
\pi}{n} \Bigr) + \sum_{j=1}^{s-1} \cos^{2m}(j \pi) \,,
\end{equation}
while \eqref{eq4d} can be expressed as
\begin{equation}\label{eq4f}
\sum_{k=0}^{q-1} \sin^{2m} \Bigl( \frac{k \pi}{n} \Bigr) =
\sum_{j=0}^{s-1} \sum_{k=jn +1}^{(j+1)n-1} \sin^{2m} \Bigl(\frac{k
\pi}{n} \Bigr) + \sum_{j=1}^{s-1} \sin^{2m}(j \pi) \,.
\end{equation}
The second sum on the rhs of \eqref{eq4e} represents a sum over
unity and hence, yields $s -1$, while the second sum on the rhs of
\eqref{eq4f} vanishes. In the first sum on the rhs of both equations
we replace $k$ by $k + nj$, where $k$ now ranges from unity to $n -
1$. Then \eqref{eq4e} and \eqref{eq4f} become
\begin{equation}\label{eq4g}
\sum_{k=0}^{q-1} \cos^{2m} \Bigl(\frac{k \pi}{n} \Bigr) =  s
\sum_{k=1}^{n-1} \cos^{2m} \Bigl(\frac{k \pi}{n} \Bigr) + s \,,
\end{equation}
and
\begin{equation}\label{eq4h}
\sum_{k=0}^{q-1} \sin^{2m} \Bigl( \frac{k \pi}{n} \Bigr) = s
\sum_{k=0}^{n-1} \sin^{2m} \Bigl(\frac{k \pi}{n} \Bigr) \,.
\end{equation}
From the definitions in Theorem \ref{main}, the sums on the rhs of
\eqref{eq4g} and \eqref{eq4h} are $C(m,n)$ and $S(m,n)$,
respectively. Moreover, by replacing $s$ by $q/n$, we arrive at the
results in the corollary, which completes the proof.
\end{proof}

\begin{corollary} \label{cor2}
If  we define the basic  power sums
$$
C(m,n,q) := \sum_{k=0}^{n-1} \cos^{2m}  \left(
\frac{qk\pi}{n}\right)\;,  \quad \mbox{and} \quad S(m,n,q):= \sum_{k=0}^{n-1}
\sin^{2m}  \left( \frac{qk\pi}{n}\right)\, ,
$$
where $n$ and $q$ are co-prime, then
\begin{equation}\label{eq5a}
C(m,n,q) = C(m,n) \, ,
\end{equation}
while
\begin{equation}\label{eq5b}
S(m,n,q)=  S(m,n)
\end{equation}
\end{corollary}

\begin{proof}
Returning to the proof of Theorem \ref{main}, we now introduce $x =
q \pi/n$, where $q$ is co-prime to $n$, into \eqref{eq3a}. (If $q$
is a negative integer, then we  take its absolute value in what
follows.) Hence \eqref{eq3a} becomes
\begin{equation}\label{eq5c}
C(m,n,q)=2^{-2m}\, n \binom{2m}{m}-1+ 2^{1-2m} \sum_{k=1}^{m} (-1)^{qk} \,\binom{2m}{m-k}\, R(kq) \cos(qk \pi/n)
\, ,
\end{equation}
where
\begin{equation*}\label{eq5d}
R(kq) = \lim_{\epsilon \to 0}
\left\{ \frac{\sin(k(q\pi +\epsilon))}{\sin(k(q \pi +\epsilon)/n)} \right\} \;.
\end{equation*}
That is, the argument of $R(k)$ has been replaced by $kq$, while the
cosine is now dependent upon $qk \pi/n$ instead of $k \pi/n$. This
means that the sum on the rhs of \eqref{eq5c} is, once again,
non-zero for all those integer values, where $k=pn$ with $p$ ranging
from unity to $\lfloor m/n \rfloor$, provided $n$ and $q$ are
co-prime. If $m < n$, the sum of the rhs of \eqref{eq5c} vanishes
and we are left with the second result in \eqref{eq2a}. Hence we
find that
\begin{equation} \label{eq5e}
(-1)^{npq} R(kq) \cos(qp \pi/n )= (-1)^{npq} (-1)^{pq(n+1)} \,n \,(-1)^{qp}=  n \,.
\end{equation}
Introducing the above result into \eqref{eq5c} yields the first result
in the corollary for $C(m,n,q)$.

To obtain \eqref{eq5b}, we put $x = q \pi/n$ in \eqref{eq3b}. Then
we arrive at
\begin{equation}\label{eq5f}
S(m,n,q)=2^{-2m}\, n \binom{2m}{m}+ 2^{1-2m} \sum_{k=1}^{m} (-1)^{k(q+1)}\binom{2m}{m-k}\,
R(kq) \cos(qk \pi/n) \, .
\end{equation}
As indicated above, the sum will only contribute when $R(kq)$ is non-zero, which occurs when $k$ is an integer multiple of $n$.
By multiplying \eqref{eq5e} throughout with the phase factor of $(-1)^{npq}$, we obtain the value of $R(kq) \cos(qk \pi/n)$, which equals
$(-1)^{qnp} n$. Introducing this value into \eqref{eq5f} yields \eqref{eq5b}. This completes the proof.
\end{proof}

Although it was stated that $n$ and $q$ need to be co-prime for
\eqref{eq5a} to hold, let us now assume that they are both even
numbers, but are co-prime once the factor of 2 has been removed. If
we let $q = 2s$ and $n = 2\ell$, then we have
\begin{equation}
C(m,n,q)= \sum_{k=0}^{n-1} \cos^{2m} \Bigl( \frac{sk \pi}{\ell} \Bigr)=2 \sum_{k=0}^{\ell-1}
\cos^{2m} \Bigl( \frac{sk \pi}{\ell} \Bigr).
\label{eq5g}\end{equation}
Since $s$ and $\ell$ are co-prime, we can apply Corollary\ 2.3. If, however, there was another factor
of 2 before $s$ and $\ell$ became co-prime, i.e.\ $n=4 \ell$ and $q=4s$, then we find that
\begin{equation}\label{eq5h}
C(m,n,q)= 4 \sum_{k=0}^{\ell-1} \cos^{2m} \Bigl( \frac{s k \pi}{\ell} \Bigr)\;.
\end{equation}
Moreover, if $r$ represents the product of all the common factors of $n$ and $q$, then we find that
\begin{equation}\label{eq5i}
\sum_{k=0}^{n-1}\left\{ \begin{matrix} \cos^{2 m} \\ \sin^{2m}  \end{matrix} \right\} \left(
\frac{qk \pi}{n} \right)= r \sum_{k=0}^{\ell-1} \left\{ \begin{matrix} \cos^{2 m} \\ \sin^{2m}
\end{matrix} \right\} \Bigl( \frac{s k \pi}{\ell} \Bigr),
\end{equation}
where the curly brackets have been introduced to signify that the above results apply to either a
cosine or sine power.Therefore, for $n = r \ell$ and $q = r s$, $C(m,n,q) = r
C(m,\ell,s) =r C(m,\ell)$  and $S(m,n,q)=r S(m,\ell,s)=rS(m,\ell)$ according to Corollary\
2.3.

\section{Extensions} \label{sec3}
In this section we shall use the results of the previous section to derive solutions for more advanced
basic trigonometric power sums. As stated in the introduction, Merca \cite{M2012} has evaluated several basic
trigonometric power sums by deriving \eqref{merca1} and \eqref{merca2} via the multisection series method. However,
these results can also be derived via Theorem \ref{main}. To demonstrate this, we express
$C(p,n)$ as
\begin{equation*}\label{eq6}
C(p,n)= 2\sum_{k=1}^{\lfloor(n-1)/2 \rfloor} \cos^{2p}(k \pi/n) + 1 \;,
\end{equation*}
where for the terms between $k = \lfloor (n-1)/2 \rfloor$ and $k = n
- 1$, we have replaced $k$ by $n-k$, thereby obtaining twice the
first sum via symmetry. In other words, the finite sum studied by
Merca in \cite{M2012} is just half of $C(p,n)$. Dividing through by
2 yields Merca's result when one realises that: (1) the sum over
negative values of $k$ in \eqref{merca1} is identical to that over
of positive values of $k$, and (2) the $k =  0$ term produces the
combinatorial term on the rhs of \eqref{eq2a}.

We can determine formulas for other basic trigonometric power sums by manipulating \eqref{eq2a} and \eqref{eq2b}. First we
express $C(m,n)$ as
$$
C(m,n) = \sum_{k=0}^{n-1} \cos^{2m}( 2k \pi/2n)=
\sum_{k=0,2,4,\ldots}^{2n-2} \cos^{2m}\left(k \pi/2n \right)\;.
$$
This result can be written alternatively as
\begin{equation}\label{eq6a}
C(m,n) = \frac{1}{2} \sum_{k=0}^{2n-1} \left( 1+ (-1)^k \right)
\cos^{2m} (k \pi/2n) \,.
\end{equation}
The first sum on the rhs is simply $C(m,2n)$. Therefore, \eqref{eq6a} becomes
\begin{equation}\label{eq7}
\sum_{k=0}^{2n-1} (-1)^k \cos^{2m}(k \pi/2n) = 2 C(m,n) - C(m,2n) \;.
\end{equation}
In fact, the above result can be extended to yield
\begin{equation*}\label{eq7a}
\sum_{k=0}^{2pn-1} (-1)^k \cos^{2m}(k \pi/2pn) = 2 C(m,pn) - C(m,2pn)
\;,
\end{equation*}
where $p$ is a positive integer. Hence the alternating form of $C(m,n)$ is given by
\begin{equation}\label{eq7b}
\sum_{k=0}^{n-1} (-1)^k \cos^{2m}(k \pi/n) = 2 C(m,n/2) - C(m,n)
\;,
\end{equation}
which is only valid for even values of $n$.

If we introduce \eqref{eq2a} into \eqref{eq7}, then we obtain three
distinct cases depending on whether $m<n$, $n \leqslant m <2n$ and
$m\geqslant 2n$. Hence we arrive at
\begin{equation}\label{eq8}
\sum_{k=0}^{2n-1} (-1)^k \cos^{2m} \Bigl( \frac{k \pi}{2n} \Bigr) =
\begin{cases} \displaystyle
2^{2-2m} \, n \Bigl( \sum_{p=1}^{\lfloor m/n
\rfloor}\binom{2m}{m-pn} -\sum_{p=1}^{\lfloor m/2n
\rfloor}\binom{2m}{m-2pn} \Bigr)\,, & m\geqslant 2 n \, ,  \cr
\displaystyle 2^{2-2m} \sum_{p=1}^{\lfloor
m/n\rfloor}\binom{2m}{m-pn}\,, & n \leqslant m<2n \,, \cr
\displaystyle 0 \, , & m<n\, .
\end{cases}
\end{equation}
In similar fashion, we can obtain the corresponding result when $\cos^{2m}(k \pi/2n)$ is replaced by $\sin^{2m}(k \pi/2n)$. Therefore,
repeating the above calculation, we obtain
\begin{equation*}\label{eq9}
\sum_{k=0}^{2n-1} (-1)^k \sin^{2m}(k \pi/2n) = 2 S(m,n) - S(m,2n)
\;,
\end{equation*}
which yields after the introduction of \eqref{eq2b}
\begin{equation}
\sum_{k=0}^{2n-1} (-1)^k \sin^{2m} \Bigl( \frac{k \pi}{2n} \Bigr) =
\begin{cases} \displaystyle
2^{2-2m} \, n \Bigl( \sum_{p=1}^{\lfloor m/n \rfloor} (-1)^{pn}
\binom{2m}{m-pn} -\sum_{p=1}^{\lfloor m/2n \rfloor}\binom{2m}{m-2pn}
\Bigr)\,, & m\geqslant 2 n \, ,  \cr \displaystyle 2^{2-2m}
\sum_{p=1}^{\lfloor m/n\rfloor}\binom{2m}{m-pn} \,, & n \leqslant
m<2n \,, \cr \displaystyle 0\, , & m<n\, . \label{eq10}
\end{cases}
\end{equation}

We can also express \eqref{eq7} as
\begin{equation}\label{eq10a}
\sum_{k=0,2,4,\ldots}^{2n-2} \cos^{2m}(k \pi/2n) -\sum_{k=1,3,5,\ldots}^{2n-1}
\cos^{2m}(k \pi/2n) = 2 C(m,n) - C(m,2n) \;.
\end{equation}
Alternatively, the above can be written as
\begin{equation*}\label{eq10b}
\sum_{k=0}^{n-1} \cos^{2m}(k \pi/n) -\sum_{k=0}^{n-1}
\cos^{2m}((k+ 1/2) \pi/n) = 2 C(m,n) - C(m,2n) \;.
\end{equation*}
The above result can be simplified further to yield
\begin{equation}\label{eq10c}
\sum_{k=0}^{n-1} \cos^{2m}((k+ 1/2) \pi/n) = C(m,2n) - C(m,n) \;.
\end{equation}
Before we can combine the terms on the rhs, we need to relate the upper
upper limit in the sum for $C(m,2n)$ in \eqref{eq10c}, viz.\ $\lfloor m/2n \rfloor$,
with that for $C(m,n)$, which is $\lfloor m/n \rfloor$. If we let $m= rn+b$, where $0<b<n$,
then $\lfloor m/n \rfloor= r$, while $\lfloor m/2n \rfloor= \lfloor r/2+b/2n\rfloor$.
If $r$ is even, then we find that $\lfloor m/n \rfloor= 2 \lfloor m/2n \rfloor$, but if
it is odd, then we find that $\lfloor m/n \rfloor = 2 \lfloor m/2n \rfloor +1$. In other
words, we require the following identity:
$$
\lfloor m/n \rfloor = 2 \lfloor m/2n \rfloor + \left(1- (-1)^{\lfloor m/n \rfloor} \right)/2 \;.
$$
We now introduce \eqref{eq2a} into \eqref{eq10c}, which yields
\begin{equation*}\label{eq10d}
\sum_{k=0}^{n-1} \cos^{2m}((k+ 1/2) \pi/n) =
2^{1-2m} \, n \biggl(\binom{2m-1}{m-1} +
2 \sum_{p=2,4, \ldots}^{2\lfloor m/2n\rfloor} \binom{2m}{m-pn} -
\sum_{p=1}^{\lfloor m/n\rfloor} \binom{2m}{m-pn} \biggr) \,.
\end{equation*}
At this stage we require the identity given above. Then we arrive at
\begin{equation}\label{eq10e}
\sum_{k=0}^{n-1} \cos^{2m} \bigl((k+ 1/2) \pi/n \bigr) =
2^{1-2m} \, n \biggl(\binom{2m-1}{m-1} +
\sum_{p=1}^{\lfloor m/n\rfloor} (-1)^p \binom{2m}{m-pn} \biggr) \,.
\end{equation}
This is basically twice Merca's result, which has been given here as \eqref{merca2}. By carrying
out a similar calculation with the cosine power in \eqref{eq7} replaced by a sine power and using
\eqref{eq2b} instead, one finds that
\begin{equation}\label{eq10f}
\sum_{k=0}^{n-1} \sin^{2m} \bigl( (k+ 1/2) \pi/n \bigr)  =
2^{1-2m} \, n \biggl(\binom{2m-1}{m-1} +
\sum_{p=1}^{\lfloor m/n\rfloor} \Bigl( 1 +(-1)^p - (-1)^{np} \Bigr) \binom{2m}{m-pn} \biggr) \,.
\end{equation}
For odd values of $n$, \eqref{eq10f} reduces to
\begin{equation*}\label{eq10fa}
\sum_{k=0}^{n-1} \sin^{2m} \bigl((k+ 1/2) \pi/n \bigr) =
2^{1-2m} \, n \biggl(\binom{2m-1}{m-1} +
\sum_{p=1}^{\lfloor m/n\rfloor} \binom{2m}{m-pn} \biggr) \,,
\end{equation*}
while for even values of $n$, it becomes
\begin{equation*}\label{eq10fb}
\sum_{k=0}^{n-1} \sin^{2m} \bigl( (k+ 1/2) \pi/n \bigr) =
2^{1-2m} \, n \biggl(\binom{2m-1}{m-1} +
\sum_{p=1}^{\lfloor m/n\rfloor} (-1)^p \binom{2m}{m-pn} \biggr) \,.
\end{equation*}

It was mentioned that Merca was able to evaluate finite sums involving
the binomial coefficient in a few corollaries by fixing $n$ to small
values ranging from unity to 5 or 6 in \eqref{merca1} and
\eqref{merca2} and directly evaluating the sums. These results can
be verified by carrying out the same procedure with the results in
Theorem \ref{main} and by using \eqref{eq10e}.

The preceding results, given by \eqref{eq8} and \eqref{eq10}, have
had a factor $\ell = 2$ introduced into the denominators of the
cosine and sine powers in the basic trigonometric power sums of
Theorem \ref{main}. We can develop other interesting results by
multiplying and dividing the argument in the trigonometric power by
integers. For example, if we multiply and divide the argument in the
cosine power by 3, then $C(m,n)$ can be expressed as
\begin{equation} \label{eq10g}
C(m,n) =\sum_{k=0,3,6,\ldots}^{3n-3} \cos^{2m} \Bigl( \frac{k \pi}{3
n} \Bigr) \;.
\end{equation}
The same applies to $S(m,n)$ when we multiply and divide the argument by 3.
To obtain a sum over all values of $k$ from 1 to $3n-1$, we need to write the
above sum as
\begin{equation}
C(m,n)= \frac{1}{3} \sum_{k=0}^{3n-1} \Bigl( 2 \cos \Bigl( \frac{2k
\pi}{3} \Bigr)+1 \Bigr) \cos^{2m} \Bigr( \frac{k \pi}{3 n} \Bigr)
\;.
\end{equation}
Consequently, we arrive at the following interesting result
\begin{equation}\label{eq11}
\sum_{k=0}^{3n-1} \cos \Bigl( \frac{2k \pi}{3} \Bigr) \cos^{2m}
\Bigr( \frac{k \pi}{3 n} \Bigr)
= \frac{1}{2} \Bigl(3\, C(m,n)
-C(m,3n) \Bigr) \,.
\end{equation}
The corresponding result for $S(m,n)$  is
\begin{equation} \label{eq12}
\sum_{k=0}^{3n-1} \cos \Bigl( \frac{2k \pi}{3} \Bigr) \sin^{2m}
\Bigr( \frac{k \pi}{3 n} \Bigr)= \frac{1}{2} \Bigl(3\, S(m,n)
-S(m,3n) \Bigr) \,.
\end{equation}
Next, by introducing the results in Theorem \ref{main} we obtain explicit
expressions for both sums, which are
\begin{equation*} \small
\sum_{k=0}^{3n-1} \cos \Bigl( \frac{2k \pi}{3} \Bigr) \cos^{2m}
\Bigr( \frac{k \pi}{3 n} \Bigr)=
\begin{cases}
\displaystyle \frac{3n}{2^{2m}} \Bigl( \sum_{p=1}^{\lfloor m/n\rfloor}\binom{2m}{m-pn}
-\sum_{p=1}^{\lfloor m/3n\rfloor}\binom{2m}{m-3pn} \Bigr) \,,& m \geqslant 3n\,, \cr
\displaystyle
\frac{3n}{2^{2m}}\sum_{p=1}^{\lfloor m/n\rfloor}\binom{2m}{m-pn} \,,& n \leqslant m< 3n\,, \cr
\displaystyle 0\;, & m<n\,,
\end{cases}
\end{equation*}
and
\begin{equation*} \small
\sum_{k=0}^{3n-1} \cos \Bigl( \frac{2k \pi}{3} \Bigr) \sin^{2m}
\Bigr( \frac{k \pi}{3 n} \Bigr)=
\begin{cases}
\displaystyle \frac{3n}{2^{2m}} \Bigl( \sum_{p=1}^{\lfloor m/n\rfloor} (-1)^{pn} \binom{2m}{m-pn}
-\sum_{p=1}^{\lfloor m/3n\rfloor} (-1)^{3pn}  & \cr
\displaystyle \times \;\; \binom{2m}{m-3pn} \Bigr) \,,& m \geqslant 3n\,, \cr
\displaystyle \frac{3n}{2^{2m}}\sum_{p=1}^{\lfloor m/n\rfloor} (-1)^{pn} \binom{2m}{m-pn} \,,& n \leqslant m< 3n\,, \cr
\displaystyle 0\;, & m<n\,.
\end{cases}
\end{equation*}
In the above results  we see that the final expressions for the sums are now
dependent on whether $m$ is greater or less than either $n$ or $3n$,
rather than $n$ and $2n$ in the previous example. Furthermore, the series on the
lhs can be written as
\begin{equation}
\sum_{k=0}^{3n-1} \cos \Bigl(\frac{2 k \pi}{3} \Bigr) \left\{
\begin{matrix} \cos^{2 m} \\ \sin^{2m}  \end{matrix} \right\} \left(
\frac{k \pi}{3n}\right) = \sum_{k=0}^{3n-1} (-1)^k \cos
\Bigl(\frac{k \pi}{3} \Bigr) \left\{ \begin{matrix} \cos^{2 m} \\
\sin^{2m}
\end{matrix} \right\} \left( \frac{k \pi}{3n}\right)\, .
\end{equation}
Nevertheless, we are unable to obtain the corresponding sum with $\cos (k \pi/3)$
in the summand instead of $\cos(2 k \pi/3)$. To accomplish that, we
need to examine the situation when we multiply and divide by
$\ell=4$ inside the trigonometric power.

For the $\ell=4$ situation the corresponding form of \eqref{eq10a}
becomes
\begin{equation*} \label{eq12a}
C(m,n) =\sum_{k=0,4,8,\ldots}^{4n-4} \cos^{2m} \Bigl( \frac{k \pi}{4
n} \Bigr) \;.
\end{equation*}
In this case the sum over all values of $k$ becomes
\begin{equation*}\label{eq13}
C(m,n)= \frac{1}{4} \sum_{k=1}^{4n-1} \Bigl( 2 \cos \Bigl( \frac{k \pi}{2} \Bigr)
+1 +(-1)^k \Bigr) \cos^{2m} \Bigl( \frac{k \pi}{4 n} \Bigr) \;.
\end{equation*}
The term involving unity in the above equation yields $C(m,4n)$,
while the term with the oscillating phase is simply \eqref{eq7}
with $n$ replaced by $2n$ or $2 C(m,2n)-C(m,4n)$. Therefore, we find
that the $C(m,4n)$ contributions cancel each other and we are left
with
\begin{equation*}\label{eq14}
\sum_{k=0}^{4n-1} \cos \Bigl( \frac{k \pi}{2} \Bigr) \cos^{2m}
\Bigl( \frac{k \pi}{4n} \Bigr)= 2 \,C(m,n) - C(m,2n) \,,
\end{equation*}
which is just another derivation of \eqref{eq7}. That is, we do not
obtain a new basic trigonometric power sum as we did when we divided
and multiplied by 3 inside the cosine power. In fact, the same
reducibility arises when we divide and multiply by 8. Therefore, we
conjecture that multiplying and dividing by $2^n$ in either $C(m,n)$
or $S(m,n)$ will not produce new formulas.

So let us now turn our attention to when we multiply and divide by 6 in
inside the cosine power of $C(m,n)$. Since we are dealing with an even number, we expect some
reducibility to occur. Then $C(m,n)$ becomes
\begin{equation*} \label{eq15}
C(m,n) =\sum_{k=0,6,12,\ldots}^{6n-6} \cos^{2m} \Bigl( \frac{k
\pi}{6 n} \Bigr) \;.
\end{equation*}
With the aid of the identity
\begin{equation} \label{eq15a}
2 \cos \Bigl( \frac{k \pi}{3} \Bigr) + 2 \cos \Bigl( \frac{2 k \pi}{3} \Bigr) +1
-(-1)^{k+1} = \begin{cases} 6\;,  &  k \equiv 0 \; ({\rm mod}\, 6)\, , \\
0\;,& {\rm otherwise}\,, \end{cases}
\end{equation}
the above equation can be written alternatively as
\begin{equation}\label{eq16}
C(m,n)= \frac{1}{6} \sum_{k=0}^{6n-1} \Bigl( 2 \cos \Bigl( \frac{k
\pi}{3} \Bigr) +2 \cos \Bigl( \frac{2k \pi}{3} \Bigr) + (-1)^k +1
\Bigr) \cos^{2m} \Bigr( \frac{k \pi}{6 n} \Bigr) \;.
\end{equation}
Expressing the sum of cosines as a product, we find that \eqref{eq16} becomes
\begin{equation*}\label{eq17}
C(m,n)= \frac{1}{6} \sum_{k=0,2,4,\ldots}^{6n-2} \Bigl( 4 \cos
\Bigl( \frac{k \pi}{2} \Bigr) \,\cos \Bigl( \frac{k \pi}{6} \Bigr) +
(-1)^k +1 \Bigr) \cos^{2m} \Bigr( \frac{k \pi}{6 n} \Bigr) \;.
\end{equation*}
Replacing $2k$ by $k$ in the above result yields
\begin{equation*}\label{eq18}
C(m,n)= \frac{1}{3} \sum_{k=1}^{3n-1} \Bigl( 2 (-1)^k \,
\cos \Bigl( \frac{k \pi}{3} \Bigr) +1 \Bigr)
\cos^{2m} \Bigr( \frac{k \pi}{3 n} \Bigr) \;.
\end{equation*}
Hence we arrive at
\begin{equation*}
\sum_{k=0}^{3n-1} (-1)^k  \,\cos \Bigl( \frac{k \pi}{3} \Bigr)\,
\cos^{2m} \Bigr( \frac{k \pi}{3 n} \Bigr) =\frac{1}{2} \, \Bigl( 3\,
C(m,n) -C(m,3n) + 3  \Bigr) \,.
\end{equation*}
This  is just \eqref{eq11} again except for the term of 3/2. Even though we have demonstrated the
reducible nature of basic trigonometric power sums, we have not been able to produce a result
with $\cos(k \pi/3)$ in the summand instead of $\cos(2 k \pi/3)$. However, it can be
seen that $\cos(k \pi/3)$ does appear in the identity given by \eqref{eq15a}. Therefore,
let us construct a situation involving the identity and the cosine power together, viz.\
\begin{equation}\label{eq19}
\frac{1}{3} \sum_{k=0}^{3n-1} \Biggl( \cos \Bigl( \frac{\pi k}{3}
\Bigr) + \cos \Bigl( \frac{2 \pi k}{3} \Bigr) +
\frac{1-(-1)^{k+1}}{2} \Biggr) \cos^{2m} \Bigl( \frac{ k \pi}{3 n}
\Bigr) = \sum_{k=0,6,12,\ldots}^{6 \lfloor(3n-1)/6 \rfloor}\cos^{2m}
\Bigl( \frac{k \pi}{3 n} \Bigr)\,.
\end{equation}
The first series on the lhs of the above equation is the result we wish to determine,
while the second series is given by \eqref{eq11}. The next term with $1/2$ is simply
$C(m,3n)/2$. Thus, we are left with two series to evaluate. The first of these can be
determined by replacing $n$ with $3n/2$ in \eqref{eq7}, which yields
\begin{equation} \label{eq20}
\sum_{k=0}^{3n-1}  (-1)^k \cos^{2m} \Bigl( \frac{k \pi}{3 n} \Bigr)=
2 C(m,3n/2) -C(m,3n) \,.
\end{equation}
Inserting the results for $C(m,n)$ in Theorem \ref{main} yields
\begin{equation} \label{eq20a} \small
\sum_{k=0}^{3n-1}  (-1)^k \cos^{2m} \Bigl( \frac{k \pi}{3 n} \Bigr)=
\begin{cases}
\displaystyle \frac{6n}{2^{2m}} \Biggl(  \sum_{p=1}^{\lfloor 2m/3n\rfloor}\binom{2m}{m-3pn/2}
-\sum_{p=1}^{\lfloor m/3n\rfloor}\binom{2m}{m-3pn} \Biggr) \,,& m \geqslant 3n\,, \cr
\displaystyle
\frac{6n}{2^{2m}}\sum_{p=1}^{\lfloor 2m/3n\rfloor}\binom{2m}{m-3pn/2} \,,& 3n/2 \leqslant m< 3n\,, \cr
\displaystyle 0\;, & m<3 n/2 \,.
\end{cases}
\end{equation}

For the above result to be valid, $n$ must also be even. Since $n$ is even,
$\lfloor (3n-1)/6  \rfloor = n/2-1$. Then the series on the rhs of \eqref{eq19} can
be expressed as
\begin{equation} \label{eq21}
\sum_{k=0,6,12,\ldots}^{6 \lfloor (3n-1)/6 \rfloor}  \cos^{2m}
\Bigl( \frac{k \pi}{3 n} \Bigr)= \sum_{k=0}^{n/2-1} \cos^{2m}
\Bigl(\frac{k \pi}{n/2} \Bigr)\,.
\end{equation}
In other words, the above sum is equal to $C(m,n/2)$ provided $n$ is even. If we
introduce \eqref{eq20} and \eqref{eq21} into \eqref{eq19} together with the other
previously mentioned results, then after a little algebra we find that
$$
\sum_{k=0}^{3n-1} \cos \Bigl( \frac{k \pi}{3} \Bigr)\, \cos^{2m}
\Bigr( \frac{k \pi}{3 n} \Bigr)= 3\, C(m,n/2) -3\, C(m,n)/2 +
C(m,3n)/2-C(m,3n/2) \, ,
$$
where $n$ is an even positive integer. Introducing the results from Theorem \ref{main}
into the above equation yields
\begin{equation*} \small
\sum_{k=0}^{3n-1} \cos \Bigl( \frac{k \pi}{3} \Bigr)\, \cos^{2m}
\Bigr( \frac{k \pi}{3 n} \Bigr)=
\begin{cases}
\displaystyle
\frac{3n}{2^{2m}} \Biggl( \sum_{p=1}^{\lfloor 2m/n\rfloor}\binom{2m}{m-pn/2}
- \sum_{p=1}^{\lfloor m/n\rfloor}\binom{2m}{m-pn}  & \cr
\displaystyle
- \sum_{p=1}^{\lfloor 2m/3n\rfloor} \binom{2m}{m-3pn/2} + \sum_{p=1}^{\lfloor m/3n \rfloor} \binom{2m}{m-3pn}\Biggr)
\,,& m \geq 3n\,, \cr
\displaystyle
\frac{3n}{2^{2m}} \Biggl( \sum_{p=1}^{\lfloor 2m/n\rfloor}\binom{2m}{m-pn/2}
- \sum_{p=1}^{\lfloor m/n\rfloor}\binom{2m}{m-pn}  & \cr
\displaystyle
- \sum_{p=1}^{\lfloor 2m/3n\rfloor} \binom{2m}{m-3pn/2}\Biggr) \,,& 3n/2 \leqslant m < 3n\,, \cr
\displaystyle
\frac{3n}{2^{2m}} \Biggl( \sum_{p=1}^{\lfloor 2m/n\rfloor}\binom{2m}{m-pn/2}
- \sum_{p=1}^{\lfloor m/n\rfloor}\binom{2m}{m-pn} \Biggr) \,,& n \leqslant m < 3n/2\,, \cr
\displaystyle
\frac{3n}{2^{2m}}\sum_{p=1}^{\lfloor 2m/n\rfloor}\binom{2m}{m-pn/2} \,,& n/2 \leqslant m< n\,, \cr
\displaystyle 0\;, & m< n/2\,.
\end{cases}
\end{equation*}

The sine version of the above basic trigonometric power sum can be obtained
in a similar manner with the various $C(m,n)$ terms replaced by their
corresponding $S(m,n)$ terms. In Appendix B we examine the case when the
argument of the cosine power is multiplied and divided by 5. In this
case two different prefactors of the cosine powers, viz., $\cos(2 \pi k/5)$
and $\cos(4 \pi k /5)$, arise, which are not easily decoupled from
one another. This implies that it is not possible without additional
information to obtain elegant results such as those above when we multiply
and divide by numbers that possess prime number factors greater than or equal to 5.

\section{Generating functions} \label{sec4}

In this section we use the results of Section \ref{sec2} to
determine several generating functions. We begin by defining the
exponential generating function
\begin{equation} \label{gf1}
G_1(n;z):=\sum_{m=0}^{\infty} \frac{z^m}{m!}\, C(m/2,n)\, .
\end{equation}
After some algebra we eventually arrive at
\begin{align}\label{gf2}
G_1(n;z)  & = \sum_{k=0}^{n-1} e^{z\cos(k\pi /n)}
=  \sum_{k=0}^{n-1} \cosh
\Bigl( z \cos \Bigl( \frac{k\pi}n \Bigr) \Bigr) + \sinh z
\nonumber \\
 & =  2n\sum_{k=0}^{\infty}\Bigl( \frac{z}{2} \Bigr)^{2k}
\sigma_k(n) +n I_0(z) + \sinh z \, ,
\end{align}
where
\begin{equation}\label{gf2a}
\sigma_k(n):= \frac 1{(2k)!}\sum_{p=1}^{\lfloor k/n \rfloor}\binom{2k}{k+pn}\, .
\end{equation}
and $I_0(z)$ represents the modified Bessel function of zeroth order. Note
that for $k < n$, $\sigma_k(n) =0$.

The result given by \eqref{gf2} can also be regarded as the generating function for
$\sigma_k(n)$. For fixed small values of $n$, the $\sigma_k(n)$ can
be determined by the direct evaluation of the series $C(m,n)$ in
Theorem \ref{main} or its variants in the corollaries. Moreover, it
was mentioned in the introduction that Merca has evaluated several
combinatorial identities involving the binomial coefficient in a
series of corollaries in  \cite{M2012}. In fact, Corollary 6 of this reference
presents specific values of the $\sigma_k(n)$ for $n$ ranging from
unity to six. These results have been determined via \eqref{merca1},
which we have seen follows from Theorem \ref{main}.

We can extend $G_1(n;z)$ to $G_1(n,q;z)$ by introducing the result for $C(m/2,n,q)$ as given
in Corollary\ \ref{cor2} into \eqref{gf1}. Then for odd values of $q$, the generating function
$G_1(n,q;z)$ becomes
\begin{align}\label{gf3}
G_1(n,q;z) & = \sum_{k=0}^{n-1} e^{z\cos(q k\pi /n)} =  \sum_{k=0}^{n-1} \cosh
\Bigl( z \cos \Bigl( \frac{qk\pi}n \Bigr) \Bigr)  +  \sinh z
\nonumber \\
& = 2n\sum_{k=0}^{\infty}\Bigl( \frac{z}{2} \Bigr)^{2k} \sigma_k(n) +n I_0(z) + \sinh z \,.
\end{align}
The intermediate member involving the summation over the hyperbolic cosine has been obtained by:
(1) expanding the exponential as a series in the first sum, (2) splitting the resultant sum into two equal
parts, (3) substituting $\cos^m(qk \pi/n)$ by $(-1)^{qm} \cos^m(\pi q(n-k)/n)$ in one of the parts and (4)
replacing $n-k$ by $k$. Moreover, the intermediate member does not apply for even values of $q$, although
the first and third members are still equal to one another.

We can also derive the generating function for the case when $C(m/2,n,q)$ is replaced by $S(m/2,n,q)$ in
the preceding analysis. That is, by defining the exponential generating function, $H_1(n,q:z)$, as
\begin{equation*} \label{gf4}
H_1(n,q;z):=\sum_{m=0}^{\infty} \frac{z^m}{m!}\, S(m/2,n,q)= \sum_{k=0}^{n-1} e^{z \sin(qk \pi/n)} \, ,
\end{equation*}
we can obtain a similar closed-form solution to \eqref{gf3}, provided $q$ is an even integer. In this instance
we introduce the result in Corollary \ref{cor2} and split the resulting sum. Then we replace $\sin^m(k\pi/n)$
by $\sin^m(\pi-k\pi/n)$ and proceed by combining the summations. Hence we
find that
\begin{align}\label{gf4a}
H_1(n,q;z) & =  n I_0(z) + 2n\sum_{k=0}^{\infty}\Bigl( \frac{z}{2} \Bigr)^{2k}
\sigma^{-}_k(n) \, ,
\end{align}
where $\sigma^{-}_k(n,q)$ has replaced $\sigma_k(n)$ and is defined as
\begin{equation*}\label{gf4b}
\sigma^{-}_k(n):= \frac{1}{(2k)!} \, \sum_{p=1}^{\lfloor k/n \rfloor} (-1)^{pn} \binom{2k}{k+pn} \,.
\end{equation*} \label{gf4c}
In obtaining \eqref{gf4a} we have split the basic sine power sum and replaced $\sin^{2m}(\pi k/n)$ in one
of the resulting sums by $\sin^{2m}(\pi +\pi(n-k)/n)$. If $n$ is even, then $\sigma^{-}_k(n)$ reduces to
$\sigma_k(n)$,  while if it is odd, then \eqref{gf4a} becomes
\begin{equation*}
\sigma_k(n,q):= \frac{1}{(2k)!} \, \sum_{p=1}^{\lfloor k/n \rfloor} (-1)^{p} \binom{2k}{k+pn} \,.
\end{equation*}
As mentioned in the introduction Merca \cite{M2012} obtains identities for the above result by fixing $n$ and
evaluating the series in \eqref{merca2} directly.

The above results are not the only examples, where the results of
Section \ref{sec2} can be used to obtain generating functions. For
example, when $|z|<1$, we can expand the denominator in the sum
$\sum_{k=1}^{n-1} 1/(1-z \cos^2(k \pi/n))$ and introduce
\eqref{eq2a}, thereby obtaining
\begin{align*}\label{gf5}
\frac{1}{n}\sum_{k=1}^{n-1}\frac 1{1-z\cos^2(k \pi/n)} = \sum_{k=0}^{\infty} \Bigl(\frac{z}{4} \Bigr)^k  \left(
\frac{(2k)!}{(k!)^2} +2 \sum_{p=1}^{\lfloor k/n \rfloor} \binom{2k}{k-pn} \right) - \frac{1}{n(1-z)}\, .
\end{align*}
The last term on the rhs, which arises from removing the $k=0$ term
in $C(m,n)$, can be incorporated as the $k=0$ term in the sum on the
lhs. In addition, by introducing the duplication formula for the
gamma function, viz., No. 8.335(1) in \cite{GR1994}, we find that
the first term on the rhs reduces to the binomial series for
$1/\sqrt{1-z}$, while we introduce the definition for
$\sigma_k(n)$ or \eqref{gf2a} into the second term on the rhs.
Consequently, we arrive at
$$
\frac{1}{n}\sum_{k=0}^{n-1}\frac 1{1-z\cos^2(k \pi/n)}=\frac{1}{\sqrt{1-z}}+2\sum_{k=0}^{\infty}\left(
\frac{z}{4} \right)^k (2k)!\,\sigma_k(n)\, .
$$
Similarly, one can consider the analogous sum where $\cos(k \pi/n)$ is replaced by $\sin(k \pi/n)$.
In this instance one employs \eqref{eq2b} in the analysis, which finally yields
$$
\frac{1}{n}\sum_{k=1}^{n-1}\frac 1{1-z\sin^2(k \pi/n)}=\frac{1}{\sqrt{1-z}}+2\sum_{k=0}^{\infty}\left(
\frac{z}{4} \right)^k (2k)!\,\sigma^{-}_k(n)\, .
$$

\section{Closed walks} \label{sec5}
Here we demonstrate that the main result of Section \ref{sec2} can
be used in calculating closed walks on a path and also in a cycle.
We begin by recalling that the adjacency matrix $A$ of a graph $G$
is the binary matrix with rows and columns indexed by the vertices
of $G$, such that the $(i,j)$-entry is equal to $1$ if $i$ and $j$
are adjacent, and zero otherwise. Since loops are not allowed in the
graphs under consideration, the diagonal entries of $A$ are all
zero. A walk of length $r$ on $G$ represents a sequence along $r+1$
adjacent vertices (not necessarily different) and hence, possesses
$r$ edges. A walk is said to be closed if the first and terminal
vertices or endpoints are the same. A circuit is known as a closed
walk when it has no repeating edges, while a closed walk with
repeating vertices is referred to as a cycle.

Evaluating the number of closed walks on a graph has been an active
topic of research that spans across combinatorics, graph theory, and linear
algebra (cf. \cite{CRS2010, HS1979, KN2013, S2013, TWKHM2013}). Although
our result for the number of closed walks will be general, when we turn to cycles,
we will need to restrict the closed walks to even length
and the cycles to odd order.

With the aid of \eqref{eq2a} we can now prove the following
theorem:
\begin{theorem} \label{thm51}
The number of closed walks of length $2m$ on a path $P_{n-1}$ is
given by
$$
 p(2m) =
\begin{cases} \displaystyle
2n\left(\binom{2m-1}{m-1}+\sum_{k=1}^{\lfloor
m/n\rfloor}\binom{2m}{m-kn}\right)-2^{2m} \, , & m\geqslant n\,, \cr
\displaystyle 2n \binom{2m-1}{m-1}-2^{2m}\;, & m<n\,,
\end{cases}
$$
\end{theorem}

\begin{proof}
It is well-known that the $(i,j)$ entry of $A^k$ represents the
number of walks on $G$ of length $k$ with endpoints $i$ and $j$.
Furthermore, if $\lambda$ is an eigenvalue of $A$, then $\lambda^k$
is an eigenvalue of $A^k$. Hence the trace of $A^k$ is equal to the
sum of the $k$th powers of the eigenvalues of $A$, which, in turn,
equals the total number of closed walks of length $k$ on $G$, which
we denote here by $p(k)$ (cf. \cite[p.14]{CRS2010}).  Because the
adjacency matrix can be represented by a tridiagonal matrix  with
ones on the sub- and super-diagonals and zeros elsewhere, its
eigenvalues for a path $P_{n-1}$ with $n-1$ vertices are
$2\cos(\ell\pi/n)$, where $\ell=1,\ldots,n-1$, (cf., e.g.,
\cite{dF2006}).  The result in the theorem follows by summing over
all values of $\ell$ and then by applying \eqref{eq2a}. This
completes the proof.
\end{proof}

It is interesting to notice that with Theorem \ref{thm51} we get the
sequence A$198632$ in \cite{L2011}.

We now turn our attention to closed walks in a cycle by presenting the
following theorem.

\begin{theorem}
The number of closed walks of length $2m$ on the cycle $C_{n}$, where
$n$ is odd, is given by
$$
 p(2m) =
\begin{cases} \displaystyle
2n \Bigl(\binom{2m-1}{m-1} +\sum_{r=1}^{\lfloor m/n\rfloor}
\binom{2m}{m-rn} \Bigr)\, , & m\geqslant n\;, \cr \displaystyle 2 n
\binom{2m-1}{m-1}\;, & m<n\;.
\end{cases}
$$
\end{theorem}

\begin{proof}
In this instance the eigenvalues of an $n$-cycle are equal to $2\cos(2\ell\pi/n)$,
for $\ell=0,1,\ldots,n-1$. Because of this, we can follow the previous proof except
that we use \eqref{eq5a} instead of \eqref{eq2a}. Consequently, we arrive
at the result in the theorem. This completes the proof.
\end{proof}

\section{Conclusion}
In this paper we have presented combinatorial forms for the two main
basic trigonometric power sums $C(m,n)$ and $S(m,n)$ in Theorem
\ref{main}. We have been able to extend these results to derive
combinatorial forms for other basic trigonometric power sums, where
either the arguments in the trigonometric power and/or their limits
have been altered. Where possible we have been able to relate our
results to existing solutions such as those appearing in
\cite{M2012}. In addition, we have demonstrated that our main
results can be applied to generating functions, but even more
interesting, is that they were used to determine the number of
closed walks on a path and in a cycle. In the future we intend to
apply the results presented here when we study the general or
twisted Dowker \cite{Do1992} and related sums \cite{Cv2012}.

\section{Acknowledgement}
This work was supported and funded by Kuwait University Research
Grant No. SM03/13. We thank the referee for alerting us to several
typos and suggesting improvements to the original manuscript.

\section{Appendix A}
In the introduction it was stated that Berndt and Yeap's result for
the sum over even powers of cotangent, viz., \eqref{eq1a}, was
imprecise and that a better formulation was given by \eqref{eq1b}.
Here, we prove this by referring to \cite{BY2002}. To enable the
reader to develop  an understanding of how polynomials in $k$ arise
when evaluating finite sums of powers of the cotangent, we shall
also show how the formula is implemented for specific values of $n$,
which is also lacking in \cite{BY2002}.

Broadly speaking, Berndt and Yeap derive \eqref{eq1b} via the third
case considered in \cite[Theorem 2.1]{BY2002}. The other two cases
will be discussed in a future work. The theorem deals with the
contour integration of the function
 $$f(z)=\cot^m(\pi z) \cot(\pi(hz-a))\cot(\pi(kz-b))$$ over a
positively oriented rectangle with vertices at $\pm i R$ and $1 \pm
iR$, where $R>\epsilon$, and possessing semi-circular indentations
at 0 and 1 of radius $\epsilon$, where $\epsilon< \min\{
(h-1+a)/h,(k-1+b)/k \}$. The third case is represented by $a=b=0$
and hence $f(z)$ becomes
\begin{equation}\label{ap1}
f(z) = (hk)^{-1} (\pi z)^{-m-2} \Bigl(\sum_{j=0}^{\infty} a(j) x^j
\Bigr)^m \sum_{\mu=0}^{\infty} a(\mu) (h' x)^{\mu}
\sum_{\nu=0}^{\infty} a(\nu) (k' x)^{\nu}   \;\;,
\end{equation}
where $a_j = (-1)^j 2^{2j} B_{2j}/(2j)!$, $x = (\pi z)^2$, $h' =
h^2$ and $k' = k^2$. From \eqref{ap1} we see that there is a pole of
order $m+2$ at  $z = 0$. Therefore, the aim is to
evaluate the residue of $f(z)$ at $z = 0$, which is given by
\begin{equation}\label{ap2}
{\rm Res} \,f(z) \Bigl{|}_{z=0}= \frac{1}{(m+1)!} \, \frac{d^{m+1}}{dz^{m+1}}\, \Bigl( z^{m+2} f(z) \Bigr) \Bigl{|}_{z=0}\;.
\end{equation}
Before we can evaluate this, we need to evaluate the product of the infinite
series on the rhs of \eqref{ap1}. Berndt and Yeap proceed by introducing coefficients
$C(j_1,\ldots,j_m,\mu,\nu)$, which are not given explicitly, but that they arise when
a sum over all $(m+2)$-tuples $(j_1,\ldots, j_m,\mu,\nu)$ is evaluated under the
condition that $2\left(\sum_{i=1}^m j_i +\mu+ \right. $ $\left.\nu \right)= m+1$.
This, however, leads to the imprecision in \eqref{eq1b}. Here we adopt a different
approach based on extending the Cauchy product formula \cite{Wei2015}.

We begin by multiplying one of the series in the power by the penultimate series in \eqref{ap1}.
If we denote this product as $P_1$, then by the Cauchy product formula, it can be expressed as
\begin{equation*} \label{ap3}
P_1= \sum_{j=0}^{\infty} x^j A_1(j) \;,
\end{equation*}
where $A_1(j) = \sum_{\ell_1=0}^{j} a(\ell_1) h^{'\,\ell_1}
a(j-\ell_1)$. Now we multiply $P_1$ by the final series in
\eqref{ap1}, which yields
\begin{equation*}\label{ap4}
P_2= P_1 \; \sum_{\nu=0}^{\infty} a(\nu) k^{'\,\nu} x^{\nu}=
\sum_{j=0}^{\infty} x^j \sum_{\ell_1=0}^{j}a(\ell_1) k^{'r\, \ell_1}
A_1(j-\ell_1) = \sum_{j=0}^{\infty} x^j A_2(j-\ell_1) \;,
\end{equation*}
and
\begin{equation*} \label{ap5}
A_2(j) = \sum_{\ell_1=0}^{j} \sum_{\ell_2=0}^{j-\ell_1} a(\ell_1) \,
a(\ell_2)\, k^{'\,\ell_1} h^{'\,\ell_2} \, a(j-\ell_1-\ell_2) \, .
\end{equation*}
Next we multiply $P_2$ by another series in the power to obtain $P_3$, obtaining
\begin{equation*}\label{ap6}
P_3= P_2 \sum_{j=0}^{\infty} a(j) x^{j}= \sum_{j=0}^{\infty} x^j
\sum_{\ell_1=0}^{j} a(\ell_1) A_2(j-\ell_1)=
\sum_{j=0}^{\infty} A_3(j) x^j\;\;,
\end{equation*}
where
\begin{equation*}\label{ap7}
A_3(j) =\sum_{\ell_1=0}^{j}\sum_{\ell_2=0}^{j-\ell_1}
\sum_{\ell_3=0}^{j-\ell_1-\ell_2} h^{'\,\ell_2}k^{'\, \ell_3}\,
a(\ell_1)\, a(\ell_2)\, a(\ell_3) \, a(j-\ell_1-\ell_2-\ell_3) \;.
\end{equation*}
Continuing this process until all the series have been multiplied out, we eventually arrive at
\begin{equation*}\label{ap8}
\Bigl(\sum_{j=0}^{\infty} a(j) x^j \Bigr)^m \sum_{\mu=0}^{\infty}
a(\mu) (h' x)^{\mu} \sum_{\nu=0}^{\infty} a(\nu) (k' x)^{\nu}  =
\sum_{j=0}^{\infty} A_{m+1}(j) \,x^j \;\;,
\end{equation*}
where the coefficients are given by
\begin{align}\label{ap9}
A_{m+1}(j) =& \sum_{\ell_1=0}^{j}\sum_{\ell_2=0}^{j-\ell_1}
\sum_{\ell_3=0}^{j-\ell_1-\ell_2} \cdots
\sum_{\ell_{m+1}=0}^{j-\ell_1-\ell_2-\cdots-\ell_{m}}
h^{'\,\ell_m}k^{'\,\ell_{m+1}}\,
\nonumber\\
& \times \;\; \prod_{i=1}^{m+1} a(\ell_i) \,
a(j-\ell_1-\ell_2-\dots -\ell_m-\ell_{m+1}) \;.
\end{align}

If we let $\ell_s=\sum_{i=1}^{m+1} \ell_{i}$ and replace the various
terms in \eqref{ap9} by their values in $f(z)$, then we find that
\begin{align}\label{ap10}
f(z) & = (\pi z)^{-m-2}\sum_{j=0}^{\infty} (-1)^j (2\pi z)^{2j}
\sum_{\ell_1,\ell_2,\ell_3,
\ldots,\ell_{m+1}=0}^{j,j-\ell_1,j-\ell_1-\ell_2,\ldots,j-\ell_s+\ell_{m+1}}
h^{'\,2\ell_m-1}\,k^{'\,2\ell_{m+1}-1}
\nonumber\\
& \times \;\;  \prod_{i=1}^{m+1} \frac{B_{2\ell_i}}{(2\ell_i)!} \,
\frac{B_{2(j-\ell_s)}}{(2(j-\ell_s))!}\;.
\end{align}
Hence there is a pole of order $m+2$ at  $z=0$. Introducing the
above result into \eqref{ap2}, we see that there is only a residue
when $m+1$ is equal to one of the even powers of $z$ inside the
summation over $j$. Therefore, $m$ must be odd for $f(z)$ to yield a
residue. Introducing \eqref{ap10} into \eqref{ap2}, with $m$
replaced by $2n-1$, where $n$ is a positive integer, yields
\begin{equation*}\label{ap11}
{\rm Res} \,f(z) \Bigl{|}_{z=0}= \frac{(-1)^n \,2^{2n} }{\pi}
\sum_{\ell_1=0,\ell_2=0, \ldots,\ell_{2n}=0}^{n,
n-\ell_1,\ldots,n-\ell_s+\ell_{2n}} h^{'\,2\ell_{2n-1}-1} k^{'\,2\ell_{2n}-1}
\prod_{i=1}^{2n} \frac{B_{2\ell_i}}{(2\ell_i)!} \,
\frac{B_{2n-2\ell_s}}{(2n-2\ell_s)!}\;,
\end{equation*}
where $\ell_s =\sum_{i=1}^{2n} \ell_i$.

To obtain a finite sum over powers of the cotangent, we need to
consider the entire contour around $f(z)$. This means that there are
simple poles at $z = (j+a)/h$ and $z = (r+b)/k$, where $j$ and $r$
are non-negative integers such that $0<j+a<h$ and $0<r+b<k$. Since
this is the third case, where $a=b=0$, these become $z=j/h$ and
$z=r/k$. Moreover, because $h=1$, we can disregard the poles at
$z=j$ , while $r$ ranges from 1 to $k - 1$. In addition, by noting
that $\lim_{y \to \infty} \cot(c(x\pm iy)+d)= \pm i$, for $c>0$ and
$d$ real, Berndt and Yeap are able to evaluate the contour integral
directly, whereby obtaining
\begin{equation*} \label{ap12}
\frac{1}{2\pi \, i}\int_C f(z) \, dz= \frac{(-1)^n}{\pi} \;.
\end{equation*}
By applying Cauchy's residue theorem, we finally arrive at \eqref{eq1b}.

To conclude this appendix, let us now discuss the implementation of
\eqref{eq1b} for $n = 2$. Then the $j_i$ range from $j_1$ to $j_4$.
For $n-j_s$ to be non-zero, we require  some of the $j_i$ to be
zero, whereas according to the Berndt-Yeap result given by
\eqref{eq1a}, they should be greater than zero. For $n =2$, $j_0$
can be equal to 0, 1, or 2. When $j_0=2$, all the other $j_i$'s must
vanish and the sum in \eqref{eq1b} contributes  the  value  $-(-1)^2
(2^4) k^3 B_4/4!$, which in turn equals $-k^3/45$ since $B_4 =
-1/30$. When $j_0=1$, either the remaining $j_i$ equal unity or
$2-j_s$ equals unity. Hence there are four possibilities, each
yielding the same contribution. The total contribution for $j_0 =1$
becomes $4 k ((-1) (2^2) B_2/2!)^2$ or $4k/9$ since $B_2 = 1/6$.
When $j_0=0$, we have two separate cases. In the first of these
cases either one of the remaining $j_i$ or $2-j_s$ equals 2. Since
there are four possibilities, we obtain a contribution of $4 \cdot
2^4 B_4/(4!\, k)$ for this case. For the second case one of the remaining
$j_i$ or $2-j_s$ is equal to unity and another one must also be
equal to unity. Since there are effectively four variables including
$2-j_s$, this means there are $\binom{4}{2}$ or 6 combinations. Therefore,
the contribution from the second case is $6 (-2^2 B_2/2!)^2/k$.
Combining the two cases yields the total contribution for $j_0=0$,
which is $(-4/45 +2/3)k^{-1}$. Thus, \eqref{eq1b} for $n=2$ gives
\begin{equation*}\label{ap13}
\frac{1}{k} \sum_{r=1}^{k-1} \cot^4 \Bigl( \frac{\pi r}{k} \Bigr)
=1- (-k^3/45 +4 k/9 +26/45 k) \;.
\end{equation*}
After a little algebra, one eventually obtains
\begin{equation*}\label{ap14}
\sum_{r=1}^{k-1} \cot^4 \Bigl( \frac{\pi r}{k} \Bigr) = \frac{1}{45}
\, (k-1) (k-2) (k^2 +3k-13)\;,
\end{equation*}
which appears as Corollary\ 2.6a in Berndt and Yeap \cite{BY2002}.
In a similar fashion one can calculate the results for $n = 3$ and
$n = 4$, the details of which are not presented here. After a
little algebra, one finds that
\begin{equation}\label{ap15}
\sum_{r=1}^{k-1} \cot^6 \Bigl( \frac{\pi r}{k} \Bigr) = \frac{1}{945} \, (k-1) (k-2) \bigr( 2k^4+6 k^3-28 k^2 -96 k
+251 \bigr) \;\;,
\end{equation}
and
\begin{eqnarray*}
\sum_{r=1}^{k-1} \cot^8 \Bigl( \frac{\pi r}{k} \Bigr) &= &
\frac{1}{14175} \, (k-1) (k-2) \bigl( 3k^6+9 k^5-59 k^4 \\
 & &- \;195 k^3+ 457 k^2 +1761 k -3551 \bigr) \;.
\end{eqnarray*}

By using a different method, Gessel has obtained \eqref{ap15},
which, aside from a phase factor, appears as $q_6(n)$ in
\cite{G1997}. It should also be mentioned that beyond $n = 4$, the
calculations become cumbersome due to the rapidly increasing number
of combinations when the $j_i$ are summed to $n$. For these values
of $n$, a computer program will be needed to evaluate \eqref{eq1b}.

\section{Appendix B}
In this appendix we consider multiplying and dividing the argument
in the trigonometric powers of the sums $C(m,n)$ and $S(m,n)$ by 5
or what is referred to as the $\ell =5$ case according to the
terminology of Section \ref{sec4}. In so doing, the material
presented here should enable the reader to consider other values of
$\ell$, although we shall see that higher values of $\ell$ are not
as tractable as the cases studied in Section \ref{sec3}.

To investigate the $\ell=5$ case, we require the following general identity:
\begin{equation}\label{b1}
\sum_{j=1}^{\ell} e^{2 \pi ij k/\ell} = \begin{cases}
\ell\,, & \quad k \equiv 0\;\; ({\rm mod}\; \ell)\, , \cr
0 \,, & \quad    \;\; {\rm otherwise}.
\end{cases}
\end{equation}
Multiplying and dividing the argument in the cosine power of
$C(m,n)$ as defined in Section \ref{sec4} by $5$, we obtain
$$
C(m,n) = \sum_{k=0,5,10,\ldots}^{5n-5} \cos^{2m} \Bigl( \frac{k \pi}{5 n} \Bigr)\,.
$$
Next we put $\ell = 5$ in \eqref{b1} and introduce it into the above
equation. After a little algebra, we arrive at
\begin{equation} \label{b2}
C(m,n)= \frac{1}{5}\sum_{k=0}^{5n-1} \Bigl( 2\, \cos \Bigl( \frac{2 \pi k}{5} \Bigr) + 2 \, \cos \Bigl( \frac{4 \pi k}{5} \Bigr) +1 \Bigr)
\cos^{2m} \Bigl( \frac{k \pi}{5 n} \Bigr) \,.
\end{equation}
The last sum on the rhs of \eqref{b2} is  $C(m,5n)$. Hence we are
left with two distinct sums. To isolate these sums, we need to
consider an even multiple of $5$, e.g., $\ell = 10$, since we
observed that the basic trigonometric sums in Section \ref{sec4}
turned out to be reducible when $\ell$ was even.

By multiplying and dividing the argument of the cosine power in
$C(m,n)$ by $10$, we find that
$$
C(m,n) = \sum_{k=0,10,20,\ldots}^{10n-10} \cos^{2m} \Bigl( \frac{k \pi}{10 n} \Bigr)\,.
$$
Now we introduce the $\ell= 10$ version of \eqref{b1} into the above
result, which after a little algebra yields
\begin{eqnarray*}
C(m,n)& = & \frac{1}{10} \sum_{k=0}^{10n-1} \Bigl( 2\, \cos \Bigl(
\frac{\pi k}{5} \Bigr) + 2\, \cos \Bigl( \frac{2 \pi k}{5} \Bigr) +
2\, \cos \Bigl( \frac{3 \pi k}{5} \Bigr)
\\
& & + \; 2 \, \cos \Bigl( \frac{4 \pi k}{5} \Bigr) +1 +(-1)^k \Bigr)
\cos^{2m} \Bigl( \frac{k \pi}{10 n} \Bigr) \,.
\end{eqnarray*}
The above result can be simplified by introducing the trigonometric
identity for the sum of two cosines, which is given as No. 1.314(3)
in \cite{GR1994}. In this instance we sum the first and fourth
cosines on the rhs and then the second and third cosines. Then we
obtain
\begin{eqnarray}
10\,  C(m,n)& = & \sum_{k=0}^{10n-1} \Bigl( 4\, \cos \Bigl(
\frac{\pi k}{2} \Bigr)  \cos \Bigl( \frac{3 \pi k}{10} \Bigr) +
 4\, \cos \Bigl( \frac{\pi k}{2} \Bigr) \cos \Bigl( \frac{\pi k}{10} \Bigr)
\nonumber\\
& & + \; 1 +(-1)^k \Bigr) \cos^{2m} \Bigl( \frac{k \pi}{10 n} \Bigr)
\,. \label{b4}
\end{eqnarray}
In \eqref{b4} all terms with odd values of $k$ vanish, so we can
replace $k$ by $2k$, which leads to
\begin{equation*} \label{b5}
10\,  C(m,n)-2\, C(m,5n) = \sum_{k=0}^{5n-1} (-1)^k \Bigl( 4\, \cos \Bigl( \frac{3 \pi k}{5} \Bigr) +
 4\, \cos \Bigl( \frac{\pi k}{5} \Bigr) \Bigr) \cos^{2m} \Bigl( \frac{k \pi}{5 n} \Bigr) \,.
\end{equation*}
Alternatively, the above result can be written as
\begin{equation} \label{b6}
10 \,C(m,n)-2\, C(m,5n) = \sum_{k=0}^{5n-1}  \Bigl( 4\, \cos \Bigl( \frac{2 \pi k}{5} \Bigr) +
 4\, \cos \Bigl( \frac{4 \pi k}{5} \Bigr) \Bigr) \cos^{2m} \Bigl( \frac{k \pi}{5 n} \Bigr) \,.
\end{equation}
This, however, is twice \eqref{b2}. Therefore, the $\ell = 10$ case
reduces to the $\ell = 5$ case, just as we observed in the $\ell =
3$ and $\ell = 6$ cases. Worse still, the two series involving
$\cos(2\pi k/5)$ and $\cos(4\pi k /5)$ cannot be decoupled. That is,
extra information is required before each of these series can be
evaluated separately. However, we can express \eqref{b2} as
\begin{equation} \label{b7}
C(m,n)= \frac{1}{5}\sum_{k=0}^{5n-1} \Bigl( 2\, \cos \Bigl( \frac{2 \pi k}{5} \Bigr) + 2 \, \cos \Bigl( \frac{6 \pi k}{5} \Bigr) +1 \Bigr)
\cos^{2m} \Bigl( \frac{k \pi}{5 n} \Bigr) \,.
\end{equation}
By applying the identity for the sum of two cosines, we finally arrive at the following basic cosine power sum:
\begin{equation}\label{b8}
\sum_{k=0}^{5n-1} \cos \Bigl( \frac{ 2\pi k}{5} \Bigr) \cos \Bigl( \frac{4 \pi k}{5} \Bigr) \cos^{2m} \Bigl( \frac{k \pi}{5 n} \Bigr)=
 \frac{1}{4} \, \left( 5 C(m,n) -C(m,5n) \right) \,.
\end{equation}
In actual fact the above result is not very surprising because the product of the cosines external to the cosine power is given by
\begin{equation*} \label{b9}
\cos \Bigl( \frac{ 2\pi k}{5} \Bigr) \cos \Bigl( \frac{4 \pi k}{5}
\Bigr) = \begin{cases} 1\,, & k \equiv 0\;,\;  ({\rm mod} \; 5)\;,
\cr -1/4 \,, & {\rm otherwise}.
\end{cases}
\end{equation*}

It is also interesting to note that we cannot use the alternating version of $C(m,n)$ to decouple the sums in \eqref{b7}. From \eqref{eq7b} we have
$$
\sum_{k=0}^{n-1} (-1)^k \cos^{2m} \Bigl( \frac{5k \pi}{5n} \Bigr)=
\sum_{k=0,5,\ldots}^{5n-5} \cos \Bigl( \frac{ \pi k}{5}\Bigr) \cos^{2m} \Bigl(\frac{k \pi}{5n} \Bigr)
=2C(m,n/2) - C(m,n)\,,
$$
where $n$ can only be even. By following \eqref{b2} we can express the above result as
$$
\frac{1}{5} \sum_{k=0}^{5n-1} \cos \Bigl( \frac{\pi k}{5}\Bigr)
\Bigl( 2 \cos \Bigl( \frac{2 \pi k}{5} \Bigr) + 2 \cos \Bigl( \frac{4 \pi k}{5} \Bigr)
+1 \Bigr) \cos^{2m} \Bigl(\frac{k \pi}{5n} \Bigr)  =2C(m,n/2) - C(m,n)\,.
$$
After a little algebra we arrive at
\begin{align*}
& \sum_{k=0}^{5n-1} \Bigl( 2 \cos \Bigl( \frac{3 \pi k}{5} \Bigr)  + 2 \cos \Bigl( \frac{ \pi k}{5} \Bigr)
\Bigr) \cos^{2m} \Bigl(\frac{k \pi}{5n} \Bigr) \\
& = \;\; 10C(m,n/2) - 2 C(m,5n/2) + C(m,5n)- 5 C(m,n) \,.
\end{align*}
Once again, we are unable to decouple the component sums. In fact, as one goes to higher primes, there will
be more component sums appearing in the final result, which makes the task of isolating them on their own
even more difficult to accomplish. Nevertheless, we can combine the cosines on the lhs, thereby obtaining
\begin{align*}\label{b11}
& \sum_{k=0}^{5n-1} \cos \Bigl( \frac{\pi k}{5} \Bigr) \cos \Bigl( \frac{ 2\pi k}{5} \Bigr) \cos^{2m} \Bigl(\frac{k \pi}{5n} \Bigr)
\nonumber\\
& = \;\;\frac{1}{4} \Bigl(  10 C(m,n/2) - 2 C(m,5n/2) + C(m,5n)- 5 C(m,n) \Bigr) \,.
\end{align*}
Comparing the above result with \eqref{b8}, we see that they are the alternating versions of one another.

We can, however, derive a result for the first sum on the rhs of \eqref{b6}, although it may not be regarded
as very elegant. First, we express the sum as
\begin{equation}\label{b12}
\sum_{k=0}^{5n-1} \cos \Bigl( \frac{2 k \pi}{5} \Bigr)  \cos^{2m} \Bigl(\frac{k \pi}{5n} \Bigr)= \sum_{k=0}^{5n-1}
 \cos \Bigl(\frac{2kn \pi}{5n} \Bigr) \cos^{2m} \Bigl(\frac{k \pi}{5n} \Bigr)\,.
\end{equation}
From \cite[No. I.1.10]{PBM2003}, we have
\begin{equation}\label{b13}
\cos \Bigl(\frac{2n k\pi}{5n} \Bigr) = 2^{2n-1} \cos^{2n} \Bigl(\frac{k \pi}{5n} \Bigr) + n \sum_{j=0}^{n-1}
\frac{(-1)^{j+1}}{j+1}\, \binom{2n-j-2}{j} 2^{2n-2j-2}\,\cos^{2n-2j-2} \Bigl(\frac{k \pi}{5n} \Bigr)\,.
\end{equation}
Introducing \eqref{b13} into \eqref{b12} yields
\begin{align*}
\sum_{k=0}^{5n-1} \cos \Bigl( \frac{2 k \pi}{5} \Bigr)  \cos^{2m}
\Bigl(\frac{k \pi}{5n} \Bigr) = & \sum_{k=0}^{5n-1} \Bigl( 2^{2n-1}
\cos^{2m+2n} \Bigl(\frac{k \pi}{5n} \Bigr)+ n \sum_{j=0}^{n-1}
\frac{(-1)^{j+1}}{j+1}\, 2^{2n-2j-2}
\nonumber\\
& \times \binom{2n-j-2}{j} \cos^{2m+2n-2j-2} \Bigl(\frac{k \pi}{5n}
\Bigr) \Bigr) \,.
\end{align*}
Recognizing that the sum over $k$ is the basic cosine power sum
defined in Section \ref{sec2}, we finally arrive at
\begin{align}\label{b16}
\sum_{k=0}^{5n-1} \cos \Bigl( \frac{2 k \pi}{5} \Bigr)  \cos^{2m}
\Bigl(\frac{k \pi}{5n} \Bigr) = & \; 2^{2n-1}\, C(m+n,5n) +n
\sum_{j=0}^{n-1}  \frac{(-1)^{j+1}}{j+1}\; 2^{2n-2j-2}
\nonumber\\
& \times \;\; \binom{2n-j-2}{j}\,  C(m+n-j-1,5n) \,.
\end{align}
The other sum involving $\cos(4 k \pi /5)$ instead of $\cos(2 k
\pi/5)$ can be directly obtained from \eqref{b6}. Although \eqref{b16} is
cumbersome, it does nevertheless demonstrate that the basic cosine power
sum given above is combinatorial in nature or rational as a consequence of Theorem
\ref{main}.


\end{document}